\documentclass[a4paper,12pt]{article}

\usepackage{amsmath,amssymb,amsthm}
\usepackage{mathabx}
\usepackage{enumitem}
\usepackage[margin=1.2in]{geometry}
\usepackage{hyperref}
\usepackage{graphicx}
\usepackage{multirow}

\newcommand{\mr}{\mathrm}
\newcommand{\mc}{\mathcal}
\newcommand{\mb}{\mathbf}
\newcommand{\bs}{\boldsymbol}
\newcommand{\mbb}{\mathbb}
\newcommand{\pp}{\partial}
\newcommand{\dualpr}[1]{\langle #1\rangle}

\renewcommand{\div}{\nabla\cdot}

\newtheorem{lemma}{Lemma}[section]
\newtheorem{theorem}{Theorem}[section]
\newtheorem{prop}{Proposition}[section]

\begin{document}

\title{A note on devising HDG+ projections on polyhedral elements}
\author{Shukai Du\thanks{Email: shukaidu@udel.edu}\,\, and Francisco-Javier Sayas\\
Department of Mathematical Sciences, University of Delaware}
\maketitle

\begin{abstract}
In this note, we propose a simple way of constructing HDG+ projections on polyhedral elements. The projections enable us to analyze the Lehrenfeld-Sch\"{o}berl HDG (HDG+) methods in a very concise manner, and make many existing analysis techniques of standard HDG methods reusable for HDG+.
The novelty here is an alternative way of constructing the projections without using $M$-decompositions as a middle step. 
This extends our previous results [S. Du and F.-J. Sayas, SpringerBriefs in Mathematics (2019)] (elliptic problems) and [S. Du and F.-J. Sayas, Math.~Comp.~89~(2020), 1745-1782] (elasticity) to polyhedral meshes. 
\end{abstract}

\section{Introduction}
The Lehrenfeld-Sch\"{o}berl HDG (HDG+) methods \cite{LeSc:2010} have recently gained considerable interest since they superconverge on polyhedral meshes in addition to the easiness of implementation.
In \cite{DuSa:2019} (elliptic problems) and \cite{DuSa_elas:2020} (elasticity), we proposed mathematical tools to incorporate the analysis of the HDG+ methods into the projection-based error analysis setting \cite{CoGoSa:2010}.
In this way, we can reuse existing analysis techniques and avoid repeated or unnecessary arguments.  In \cite{DuSa:2019} and \cite{DuSa_elas:2020}, the projections were devised for simplicial elements. In this paper, we extend the results to polyhedral elements.

To motivate the discussion, let us review some existing works. For mixed finite element methods (or simply mixed methods), the core in their design and analysis is the local projection operators; see, for instance, \cite{RaTh:1977} for the Raviart-Thomas (RT) projection, \cite{BrDoMa:1985} for the Brezzi-Douglas-Marini (BDM) projection, and \cite{Ne:1980,Ne:1986} for the N\'{e}d\'{e}lec projection.
These projections satisfy certain commutativity properties that can be used to analyze the numerical methods in a very concise way.
Inspired by the mixed method projections, the first HDG projection was devised in \cite{CoGoSa:2010}. It enables us to analyze a wide class of HDG methods in an unified, and also simple and concise manner. Since for both the mixed methods and the HDG methods, the core in their error analysis is the specially devised projections that are tailored to the numerical schemes, this way of analysis is often referred to as the ``projection-based error analysis'' (PBEA).

PBEA has been widely used to analyze HDG methods. See, for instance, the error analysis of the HDG methods for heat/fractional diffusion \cite{ChCo:2012,CoMu:2015}, acoustic waves \cite{CoQu:2014,CoFuHuJiSaSa:2018}, Stokes equations \cite{CoGoNgPeSa:2011}, Helmholtz equations \cite{GrMo:2011}. On the other hand, new HDG projections have been devised, incorporating more variants of HDG methods into the PBEA setting; see the work of $M$-decompositions \cite{CoFuSa:2017}, an mathematical tool to systematically devise superconvergent HDG methods on polyhedral meshes. Since all $M$-decompositions HDG methods have associated HDG projections, all of their analysis can be incorporated into the PBEA setting.

Despite the wide and successful applications of HDG projections, the error analysis of some important HDG methods can {\sl not} be incorporated into the PBEA setting until very recently. An important example is the HDG+ method, proposed first by Lehrenfeld and Sch\"{o}berl \cite{LeSc:2010} and then analyzed by Oikawa \cite{Oi:2014} in the setting of elliptic diffusion. The method uses $P_k^d$-$P_{k+1}$-$P_k$ to approximate the flux-primal-trace triplet, and it achieves optimal convergence for all variables on general polyhedral meshes. Compared to the standard $P_k^d$-$P_k$-$P_k$ HDG method, the HDG+ method is as efficient as the standard method, since the two methods share the same size of the global systems. Moreover, the HDG+ method does not suffer from the problem of losing convergence order, which is observed for the standard HDG method on non-simplicial polyhedral meshes, or for elastic problems with strong symmetric stress formulation. Finally, the implementation of the HDG+ method is straight-forward, since it can be regarded as a simple tweak of the standard HDG method.

As is mentioned before, most of the existing error analysis of the HDG+ methods (see, for instance, \cite{HuPrSa:2017,Oi:2014,QiShSh:2018,QiSh:2016}) can not be incorporated into the PBEA setting. This makes their error analysis less concise compared to those HDG methods that can be analyzed by HDG projections. More importantly, this leads to a scattered style of error analysis and makes it hard for us to reuse the existing projection-based analysis techniques that were established in a decade.  All the above indicates the necessity to develop mathematical tools to incorporate the error analysis of HDG+ methods into the PBEA setting. In this way, many existing works using HDG projections, such as the analysis of the HDG methods for various types of evolutionary equations and Helmholtz equations (see, for instance, \cite{ChCo:2012,CoMu:2015,CoQu:2014,GrMo:2011,CoFuHuJiSaSa:2018}), can be automatically reused for the design and analysis of the HDG+ methods.

Following this idea, we have devised the HDG+ projections in \cite{DuSa:2019} for elliptic problems and in \cite{DuSa_elas:2020} for elasticity with strong symmetric stress formulation. We have sucessfully used the projections, combined with some existing analysis techniques of the standard HDG methods, to derive the error estimates of the HDG+ methods for heat diffusion and acoustic waves in \cite{DuSa:2019}, and for time-harmonic and transient elastic waves in \cite{DuSa_elas:2020}. For simplicity, we have limited the discussions on simplicial meshes in \cite{DuSa:2019,DuSa_elas:2020}. In this paper, we extend the results to polyhedral meshes by using an alternative way of constructing the projections without using $M$-decompositions \cite{CoFuSa:2017} as a middle step.

We finally give an outline for the rest of the paper. 
In Section \ref{sec:hdgp_pj_ell}, we devise the HDG+ projection for elliptic problems.
We also demonstrate how to use the projection to analyze the HDG+ method for a model problem.
In Section \ref{sec:hdgp_pj_elas}, we devise the HDG+ projection for elasticity. We will not demonstrate its usage, since this has been done in \cite{DuSa_elas:2020}. The projection we devise here satisfies \cite[Theorem 2.1]{DuSa_elas:2020} and it will render all the analysis and estimates in \cite[Sections 5,6\&7]{DuSa_elas:2020} valid for general polyhedral meshes.

\section{The projection for elliptic problems}
\label{sec:hdgp_pj_ell}
In this section, we devise the HDG+ projection and demonstrate how to use it to derive the error estimates for the HDG+ method. Note that the first analysis of the HDG+ method was obtained in \cite{Oi:2014}. However, our proof here is quite different from the proof in \cite{Oi:2014}. Instead, as we will demonstrate in Section \ref{sec:prf_est}, the proof we obtained is very similar to those used in \cite{CoGoSa:2010}, thanks to the introduction of the HDG+ projection. In this way, we are able to reuse the existing projection-based error analysis to analyze the HDG+ method in a very concise way. Consequently, we can unify the analysis of the standard HDG and HDG+ methods.

\quad\\
\noindent{\bf Notation.}
Let us first introduce some notation that will be used throughout the paper. 
Let $\Omega\subset\mbb R^d$ ($d=2,3$) be a polyhedral domain with Lipschitz continuous boundary. We consider a triangulation of $\Omega$ denoted by $\mc T_h$, where each $K\in\mc T_h$ is a star-shaped polyhedron. We use the standard notation $h_K$ as the diameter of $K$. Let $\mc E_K$ and $\mc E_h$ denote the collections of all the faces of $K$ and $\mc T_h$, respectively. We write $h:=\max_{K\in\mc T_h}h_K$ as the mesh-size and $h_\mr{min}:=\min_{K\in\mc T_h}h_K$ as the smallest diameter among all elements.

Let $\mc P_k(X)$ denote the polynomial space of degree $k$ on $X$ and let $\Pi_k:L^2(X)\rightarrow\mc P_k(X)$ and $\bs\Pi_k:L^2(X)^d\rightarrow\mc P_k(X)^d$ be the corresponding $L^2$ projections. Here $X$ can be an element $K$ or a face of $K$. Let $\mc R_k(\pp K):=\prod_{F\in\mc E_K}\mc P_k(F)$ and let $\mr P_M: \prod_{K\in\mc T_h}L^2(\pp K)\rightarrow\prod_{K\in\mc T_h}\mc R_k(\pp K)$ be the corresponding $L^2$ projection. 
We finally introduce the following notation for the discrete inner products on $\mc T_h$ and $\pp\mc T_h$:
\begin{align*}
(*_1,*_2)_{\mc T_h}=\sum_{K\in\mc T_h}(*_1,*_2)_K,\quad
\dualpr{*_1,*_2}_{\pp\mc T_h}=\sum_{K\in\mc T_h}\dualpr{*_1,*_2}_{\pp K},
\end{align*}
where $(\cdot,\cdot)_K$ and $\dualpr{\cdot,\cdot}_{\pp K}$ denote the $L_2$ inner products on $K$ and $\pp K$, respectively.  

\quad\\
\noindent{\bf Model problem.}
In this section, we consider the following steady-state diffusion equations:
\begin{subequations}\label{eq:ell_pde}
\begin{alignat}{5}
\kappa^{-1}\mb q+\nabla u &= 0 &\quad& \mr{in}\ \Omega,\\
\div\mb q &= f && \mr{in}\ \Omega,\\
u &= g && \mr{on}\ \Gamma:=\pp\Omega,
\end{alignat}
\end{subequations}
where the parameter $\kappa\in L^\infty(\Omega)$ is uniformly positive, the forcing term $f\in L^2(\Omega)$ and the Dirichlet data $g\in H^{\frac{1}{2}}(\Gamma)$. 
We introduce a regularity condition:
\begin{align}\label{eq:ell_reg}
\|\kappa\nabla\phi\|_{r_0,\Omega}+\|\phi\|_{1+r_0,\Omega}\le C_\mr{reg} \|\div(\kappa\nabla\phi)\|_\Omega
\end{align}
holds for any $\phi\in H_0^1(\Omega)$ such that the right term of the above inequality is finite, where $r_0\in(\frac{1}{2},1]$ is a fixed index and $C_\mr{reg}$ is a positive constant depending only on $r_0$, $\kappa$, and $\Omega$.

\quad\\
\noindent{\bf Shape-regularity of the meshes.} For each $K\in\mc T_h$, we assume the number of the faces of $K$ is bounded by a fixed constant. 
We define the shape-regularity constant of $K$, denoted as $\gamma_K$, as the minimal value $\gamma$ satisfying the following conditions (see \cite{BrSc:2008,DiDr:2017,Oi:2014} for more on shape-regularity of polyhedral elements):
\begin{itemize}
\item Chunkiness condition: $K$ is star-shaped with respect to a ball with the radius $\rho_K$ such that $\frac{h_K}{\rho_K}\le\gamma$.
\item Simplex condition: $K$ admits a simplex decomposition such that for each simplex $T$, if $h_T$ is the diameter of $T$ and $\rho_T$ is the inradius, then $\frac{h_T}{\rho_T}\le\gamma$.
\item Local quasi-uniformity: 
Let $A^\mr{max}$ and $A^\mr{min}$ be the areas of the largest and the smallest faces of $K$, respectively, then $\frac{A^\mr{max}}{A^\mr{min}}\le\gamma$.
\end{itemize}
We assume that there is a fixed positive constant $\gamma_0$ such that $\gamma_0\ge\gamma_K$ for all $K\in\mc T_h$ (consequently the shape-regularity of $\mc T_h$ is controlled).  

\quad\\
\noindent{\bf HDG+ method.}
Let us first define the approximation spaces:
\begin{align*}
\mb V_h:=\prod_{K\in\mc T_h}\mc P_k(K)^d,\quad
W_h:=\prod_{K\in\mc T_h}\mc P_{k+1}(K),\quad
M_h:=\prod_{F\in\mc E_h}\mc P_k(F).
\end{align*}
The HDG+ scheme is defined as follows: 
find $(\mb q_h,u_h,\widehat{u}_h)\in \mb V_h\times W_h\times M_h$ such that
\begin{subequations}\label{eq:HDG_s}
\begin{align}
\label{eq:HDG_s_1}
(\kappa^{-1}\mb q_h,\mb r)_{\mc T_h} - (u_h,\div\mb r)_{\mc T_h} + \dualpr{\widehat{u}_h,\mb r\cdot\mb n}_{\pp\mc T_h} &= 0,\\
\label{eq:HDG_s_2}
(\div\mb q_h,w)_{\mc T_h} + \dualpr{\tau\mr P_M(u_h-\widehat{u}_h),w}_{\pp\mc T_h}
&=(f,w)_{\mc T_h},\\
\label{eq:HDG_s_3}
-\dualpr{\mb q_h\cdot\mb n+\tau(u_h-\widehat{u}_h),\mu}_{\pp\mc T_h\backslash\Gamma} &= 0,\\
\label{eq:HDG_s_4}
\dualpr{\widehat{u}_h,\mu}_{\Gamma} &= \dualpr{g,\mu}_\Gamma,
\end{align}
\end{subequations}
for all $(\mb r,w,\mu)\in\mb V_h\times W_h\times M_h$. The stabilization function 
$\tau\in\prod_{K\in\mc T_h}\mc R_0(\pp K)$ and it satisfies $c_1h_K^{-1}\le \tau\big|_{\pp K}\le c_2h_K^{-1}$ for all $K\in\mc T_h$, where $c_1$ and $c_2$ are two fixed positive constants. 

\subsection{Main results}
We now present the main results in this section -- the HDG+ projection (Theorem \ref{thm:hdgp_pj_ell}) and its application (Theorem \ref{thm:est_ell}).
Their proofs can be found in Section \ref{sec:ell_hdg+_pfs} and Section \ref{sec:prf_est}.

\quad\\
\noindent{\bf HDG+ projection.}
\begin{subequations}\label{eq:def_hdg+_ell}
For any $K\in\mc T_h$, the HDG+ projection is defined as follows:
\begin{align*}
\bs\Pi^\mr{H+}: H^{\frac{1}{2}+\epsilon}(K)^d\times H^{\frac{1}{2}+\epsilon}(K)&\rightarrow \mc P_k(K)^d\times\mc P_{k+1}(K),\\
(\mb q,u)&\mapsto (\bs\Pi^\mr{H+}\mb q,\Pi^\mr{H+}u),
\end{align*}
where the first component $\bs\Pi^\mr{H+}$ is defined by solving
\begin{alignat}{5}\label{eq:def_hdg+_ell_2}
(\bs\Pi^\mr{H+}\mb q-\mb q,\mb r)_K &=0 &\quad &\forall\mb r\in\big(\nabla\mc P_{k+1}(K)\big)^{\perp_k},\\
\label{eq:def_hdg+_ell_3}
(\bs\Pi^\mr{H+}\mb q-\mb q,\nabla w)_K&=\dualpr{\mr P_M(\mb q\cdot\mb n)-\mb q\cdot\mb n,w}_{\pp K}
&&\forall w\in \mc P_{k+1}(K),
\end{alignat}
and the second component $\Pi^\mr{H+}u:=\Pi_{k+1}u$, namely the $L^2$ projection to $\mc P_{k+1}(K)$.
In the above equations, $(\cdot)^{\perp_k}$ represents the orthogonal complement in the background space $\mc P_k(K)^d$ and $\epsilon$ is any small positive value.

We also define an operator:
\begin{align}\label{eq:def_hdg+_ell_4}
\delta_{\pm\tau}^{\Pi^\mr{H+}}(\mb q,u):=\bs\Pi^\mr{H+}\mb q\cdot\mb n-\mr P_M(\mb q\cdot\mb n)
\pm\tau(\mr P_M\Pi^\mr{H+} u-\mr P_Mu)\in\mc R_k(\pp K).
\end{align}
\end{subequations}
We call $\delta_{\pm\tau}^{\Pi^\mr{H+}}$ the {\sl boundary remainder} of $\bs\Pi^\mr{H+}$.

\begin{theorem}[HDG+ projection]\label{thm:hdgp_pj_ell} 
For $(\mb q,u)\in H^{\frac{1}{2}+\epsilon}(K)^d\times H^{\frac{1}{2}+\epsilon}(K)$, the projection $\bs\Pi^\mr{H+}$ and the boundary remainder $\delta_{\pm\tau}^{\Pi^\mr{H+}}(\mb q,u)$ are well defined by \eqref{eq:def_hdg+_ell} and they satisfy
\begin{subequations}\label{eq:hdgp_pj_ell_eqn}
\begin{alignat}{5}
\label{eq:hdgp_pj_ell_eqn_1}
(\Pi^\mr{H+}u - u,v)_K &= 0\qquad\,\forall v\in\mc P_{k-1}(K),\\
\label{eq:hdgp_pj_ell_eqn_2}
\dualpr{\bs\Pi^\mr{H+}\mb q\cdot\mb n-\mb q\cdot\mb n\pm\tau(\Pi^\mr{H+} u-u),\mu}_{\pp K} &= 
\dualpr{\delta_{\pm\tau}^{\Pi^\mr{H+}}(\mb q,u),\mu}_{\pp K}\\
\nonumber&\qquad\qquad\forall \mu\in\mc R_k(\pp K),\\
\label{eq:hdgp_pj_ell_eqn_3}
(\div(\bs\Pi^\mr{H+}\mb q-\mb q),w)_K\pm\dualpr{\tau\mr P_M(\Pi^\mr{H+}u-u),w}_{\pp K}
&= \dualpr{\delta_{\pm\tau}^{\Pi^\mr{H+}}(\mb q,u), w}_{\pp K}\\
\nonumber&\qquad\qquad\forall w\in\mc P_{k+1}(K).
\end{alignat}
\end{subequations}
Furthermore,
\begin{subequations}
\label{eq:hdgp_pj_ell_eqn_conv}
\begin{align}
\|\bs\Pi^\mr{H+}\mb q-\mb q\|_K &\le Ch_K^{m_1}|\mb q|_{m_1,K},\\
\|\Pi^\mr{H+}u-u\|_K &\le Ch_K^{m_2}|u|_{m_2,K},\\
\|\tau^{-\frac{1}{2}}\delta_{\pm\tau}^{\Pi^\mr{H+}}(\mb q,u)\|_{\pp K}
&\le C\left(h_K^{m_1}|\mb q|_{m_1,K}+h_K^{m_2-1}|u|_{m_2,K}\right),
\end{align}
where $m_1\in[\frac{1}{2}+\epsilon,k+1]$ and $m_2\in[\frac{1}{2}+\epsilon,k+2]$. Here, the constant $C$ depends only on $k$, $\gamma_K$, and $c_2$.
\end{subequations}
\end{theorem}

Note that in Theorem \ref{thm:hdgp_pj_ell}, equations \eqref{eq:hdgp_pj_ell_eqn} do not define the HDG+ projection. However, they are exactly what we need for the error analysis. Given a projection $\bs\Pi$, the boundary remainder operator $\delta_{\tau}^{\Pi}$ can be regarded as an indicator for a kind of ``conformity'' of the projection.
For instance, if $\bs\Pi$ is the classical HDG projection \cite{CoGoSa:2010}, then we have $\delta_\tau^{\Pi^\mr{HDG}}=0$. This can be easily obtained by using \cite[Eqn. (2.1c)]{CoGoSa:2010}. Similarly, we have $\delta_{\tau=0}^{\Pi^\mr{RT}}=0$ and $\delta_{\tau=0}^{\Pi^\mr{BDM}}=0$, where $\bs\Pi^\mr{RT}$ and $\bs\Pi^\mr{BDM}$ represent the Raviart-Thomas and the BDM projection, respectively.

The key idea behind the HDG+ projection is to find weaker but still sufficient conditions to carry out a projection-based error analysis. For the classical HDG projection, the boundary remainder is zero, and the equations that define the projection are also the equations that we use for the error analysis. However, these two properties are not necessary, especially if we want to extend the projection-based error analysis to more variants of HDG methods. Taking the HDG+ method as an example, the guideline for devising the projection now becomes the following: among all the projections that satisfy the equations \eqref{eq:hdgp_pj_ell_eqn}, find one such that its approximation property is optimal, and its boundary remainder is as small as possible. As we will see soon, there is no need to enforce the boundary remainder to be zero, which is the case the standard HDG projection. In fact, a small enough boundary remainder is sufficient for optimal convergence of the method.
In this way, we can devise HDG projections more flexibly, and generalize the classical projection-based error analysis of HDG methods \cite{CoGoSa:2010}.

\quad\\
{\bf Error estimates.} By using \eqref{eq:def_hdg+_ell}, we define the element-wise projections and the boundary remainder of the exact solutions $(\mb q,u)$ defined by \eqref{eq:ell_pde}:
\begin{alignat*}{5}
&\bs\Pi^\mr{H+}\mb q:=\prod_{K\in\mc T_h}\bs\Pi^\mr{H+}\mb q,&\quad&
\Pi^\mr{H+} u:=\prod_{K\in\mc T_h}\Pi^\mr{H+}u,&\quad&
\delta_\tau^{\Pi^\mr{H+}}(\mb q,u):=\prod_{K\in\mc T_h}\delta_\tau^{\Pi^\mr{H+}}(\mb q,u).
\end{alignat*}
We also define the norm $\|\cdot\|_h$ by $\|\mu\|_h^2:=\sum_{K\in\mc T_h}h_K\|\mu\|_{\pp K}^2$ for any $\mu\in L^2(\pp\mc T_h):=\prod_{K\in\mc T_h}L^2(\pp K)$.

\begin{theorem}\label{thm:est_ell}
For the HDG+ solution $(\mb q_h,u_h,\widehat{u}_h)$ defined by \eqref{eq:HDG_s} and the exact solution $(\mb q,u)$ defined by \eqref{eq:ell_pde}, we have
\begin{align}\label{eq:est_qh}
\|\bs\Pi^\mr{H+}\mb q-\mb q_h\|_{\mc T_h}
&\le C_1\left(
\|\bs\Pi^\mr{H+}\mb q-\mb q\|_{\mc T_h}+\|\tau^{-\frac{1}{2}}\delta_\tau^{\Pi^\mr{H+}}(\mb q,u)\|_{\pp\mc T_h}
\right).
\end{align}
If the regularity condition \eqref{eq:ell_reg} holds, then we have
\begin{align}\label{eq:est_uh}
\|\Pi^\mr{H+} u-u_h\|_{\mc T_h}\le
C_2\,h^{r_0}\left(
\|\bs\Pi^\mr{H+}\mb q-\mb q\|_{\mc T_h}+\|\tau^{-\frac{1}{2}}\delta_\tau^{\Pi^\mr{H+}}(\mb q,u)\|_{\pp\mc T_h}+Q_{k}
\right),\\
\label{eq:est_uhhat}
\|\mr P_Mu-\widehat{u}_h\|_h\le C_2\,h(1+\frac{h^{r_0}}{h_\mr{min}})\left(
\|\bs\Pi^\mr{H+}\mb q-\mb q\|_{\mc T_h}+\|\tau^{-\frac{1}{2}}\delta_\tau^{\Pi^\mr{H+}}(\mb q,u)\|_{\pp\mc T_h}
+Q_{k}
\right),
\end{align}
where $Q_{k}=0$ if $k\ge1$ and $Q_{k}=\|h_K^{\frac{1}{2}}(\bs\Pi_0\mb q-\mb q)\|_{\pp\mc T_h}$ if $k=0$.
Here, $C_1$ depends only on $\kappa$, and $C_2$ depends additionally on $k$, $\gamma_0$, and $C_\mr{reg}$. 
\end{theorem}

We make some remarks about Theorem \ref{thm:est_ell}.
\begin{itemize}
\item By \eqref{eq:est_qh} and \eqref{eq:hdgp_pj_ell_eqn_conv}, we know
$\|\mb q-\mb q_h\|_{\mc T_h}$ converges optimally in the sense that
\begin{align*}
\|\mb q-\mb q_h\|_{\mc T_h}\lesssim h^m(|\mb q|_{m,\mc T_h}+|u|_{m+1,\mc T_h})\quad \forall m\in[\frac{1}{2}+\epsilon,k+1].
\end{align*}
\item If \eqref{eq:ell_reg} holds, then by \eqref{eq:est_uh}, \eqref{eq:hdgp_pj_ell_eqn_conv}, and \eqref{eq:bs_inq_2} we have
\begin{align*}
\|u-u_h\|_{\mc T_h}\lesssim h^{m+r_0}(|\mb q|_{m,\mc T_h}+|u|_{m+1,\mc T_h})\quad \forall m\in[\frac{1}{2}+\epsilon,k+1].
\end{align*}
Specifically, if $r_0=1$, namely, the elliptic regularity holds, then $\|u-u_h\|_{\mc T_h}$ achieves optimal convergence.
Since the global system is only about $\widehat{u}_h\big|_F\in\mc P_k(F)$, it can be regarded that $u_h$ achieves superconvergence without post-processing in comparison to the standard HDG method, for which a post-processing is needed to achieve an additional order of convergence for $u_h$.

\end{itemize}

Theorem \ref{thm:est_ell} can be  proved by adopting a very similar analysis used in \cite{CoGoSa:2010}, combined with the HDG+ projection. We show how this is done in Section \ref{sec:prf_est}.

\subsection{Proof of Theorem \ref{thm:hdgp_pj_ell}}
\label{sec:ell_hdg+_pfs}
In this subsection, we prove Theorem \ref{thm:hdgp_pj_ell}. We begin by presenting a lemma that gives a collection of lifting/inverse inequalities and convergence properties about $L^2$ projections. These inequalities will be used extensively in the paper.

\begin{lemma}\label{lm:bs_inq}
If $u\in\mc P_k(K)$, then
\begin{subequations}\label{eq:bs_inq}
\begin{align}
\label{eq:bs_inq_1}
\|u\|_{\pp K}\le C h_K^{-\frac{1}{2}} \|u\|_K,\quad
\|\nabla u\|_K\le C h_K^{-1}\|u\|_K.
\end{align}
For $u\in H^{\frac{1}{2}+\epsilon}(K)$, we have
\begin{align}
\label{eq:bs_inq_2}
|\Pi_ku-u|_{m,K}\le C h_K^{s-m}|u|_{s,K},\quad
\|\Pi_ku-u\|_{\pp K}\le C h_K^{t-\frac{1}{2}}|u|_{t,K}.
\end{align}
\end{subequations}
Here, $s\in[0,k+1]$, $t\in[\frac{1}{2}+\epsilon,k+1]$, $m\in\{0,\min(1,s)\}$, and the constant $C$ depends only on $k$ and $\gamma_K$.
\end{lemma}
\begin{proof}
\eqref{eq:bs_inq_1} can be obtained by \cite[Lemma 1.28]{DiDr:2020} and \cite[Lemma 1.32]{DiDr:2020} (also using our assumption that the number of the faces for each element is bounded), \eqref{eq:bs_inq_2} with $t\ge1$ can be found in
\cite[Theorem 1.45]{DiDr:2020} (see also Remark 1.49). The second inequality of \eqref{eq:bs_inq_2} in the case of $\frac{1}{2}+\epsilon\le t<1$ can be obtained by applying \cite[Lemma 7.2]{ErGu:2017} to the term $\|\Pi_ku-u\|_{\pp K}$, which gives
\begin{align*}
\|\Pi_ku-u\|_{\pp K}\lesssim h_K^{-\frac{1}{2}}\|\Pi_k u-u\|_K+h_K^{t-\frac{1}{2}}|\Pi_ku-u|_{t,K},
\end{align*}
and then use \cite[Theorem 1.45]{DiDr:2020} again to estimate the right-hand-side terms.
\end{proof}

We next prove that the projection $\bs\Pi^\mr{H+}$ is well defined by \eqref{eq:def_hdg+_ell_2} and \eqref{eq:def_hdg+_ell_3}, and it converges optimally.

\begin{prop}\label{prop:g+}
The projection $\bs\Pi^\mr{H+}$ is well defined by \eqref{eq:def_hdg+_ell_2} and \eqref{eq:def_hdg+_ell_3}, and we have
\begin{align}\label{eq:g+_conv}
h_K^{\frac{1}{2}}\|\bs\Pi^\mr{H+}\mb q-\mb q\|_{\pp K}+\|\bs\Pi^\mr{H+}\mb q-\mb q\|_K\le 
Ch_K^m|\mb q|_{m,K},
\end{align}
where $m\in[\frac{1}{2}+\epsilon,k+1]$. Here, the constant $C$ depends only on $k$ and $\gamma_K$.
\end{prop}

\begin{proof}
In this proof, we use the sign `$\lesssim$' to hide a constant that depends only on $k$ and $\gamma_K$.  First note that \eqref{eq:def_hdg+_ell_2} and \eqref{eq:def_hdg+_ell_3} define a square system. We next prove the convergence equation \eqref{eq:g+_conv}, from which the unique solvability of \eqref{eq:def_hdg+_ell_2} and \eqref{eq:def_hdg+_ell_3} follows automatically.  Let $\bs\varepsilon_q:=\bs\Pi^\mr{H+}\mb q-\bs\Pi_k\mb q\in\mc P_k(K)^d$.
By \eqref{eq:def_hdg+_ell_2} and \eqref{eq:def_hdg+_ell_3}, we have
\begin{subequations}
\begin{alignat}{5}\label{eq:cv_p_0}
(\bs\varepsilon_q,\mb r)_K &=0 &\quad& \forall\mb r\in\big(\nabla\mc P_{k+1}(K)\big)^{\perp_k},\\
\label{eq:cv_p_3}
(\bs\varepsilon_q,\nabla w)_K&=\dualpr{\mr P_M(\mb q\cdot\mb n)-\mb q\cdot\mb n,w}_{\pp K}
 && \forall w\in \mc P_{k+1}(K).
\end{alignat}
\end{subequations}
We now decompose $\bs\varepsilon_q$ into the summation $\bs\varepsilon_q=\bs\varepsilon_q^1+\bs\varepsilon_q^2$, where $\bs\varepsilon_q^1\in\nabla\mc P_{k+1}(K)$ and $\bs\varepsilon_q^2\in\big(\nabla\mc P_{k+1}(K)\big)^{\perp_{k}}$.  By \eqref{eq:cv_p_0} we have
$\|\bs\varepsilon_q\|_K^2=(\bs\varepsilon_q,\bs\varepsilon_q^1)_K$.
Since $\bs\varepsilon_q^1\in\nabla\mc P_{k+1}(K)$, we can write $\bs\varepsilon_q^1=\nabla(p+c)$ for some $p\in\mc P_{k+1}(K)$ and arbitrary constant $c$. This with \eqref{eq:cv_p_3} gives
\begin{align*}
\|\bs\varepsilon_q\|_K^2&=(\bs\varepsilon_q,\nabla(p+c))_K=\dualpr{\mr P_M(\mb q\cdot\mb n)-\mb q\cdot\mb n,p+c}_{\pp K}\\
&\lesssim h_K^{-\frac{1}{2}}\|\mr P_M(\mb q\cdot\mb n)-\mb q\cdot\mb n\|_{\pp K}\|p+c\|_{K}.
\end{align*}
We now choose the constant $c=-\Pi_0 p$ and obtain
\begin{align*}
\|\bs\varepsilon_q\|_K^2\lesssim h_K^{\frac{1}{2}}\|\mr P_M(\mb q\cdot\mb n)-\mb q\cdot\mb n\|_{\pp K}\|\nabla p\|_K\le h_K^{\frac{1}{2}}\|\mr P_M(\mb q\cdot\mb n)-\mb q\cdot\mb n\|_{\pp K}\|\bs\varepsilon_q\|_K,
\end{align*}
by which the estimate of the volumetric term $\|\bs\Pi^\mr{H+}\mb q-\mb q\|_K$ in \eqref{eq:g+_conv} is obtained. For the boundary term $\|\bs\Pi^\mr{H+}\mb q-\mb q\|_{\pp K}$, note that
\begin{align*}
\|\bs\Pi^\mr{H+}\mb q-\mb q\|_{\pp K}\le \|\bs\varepsilon_q\|_{\pp K}+\|\bs\Pi_k\mb q-\mb q\|_{\pp K}\lesssim h_K^{-\frac{1}{2}}\|\bs\varepsilon_q\|_K+\|\bs\Pi_k\mb q-\mb q\|_{\pp K},
\end{align*}
where we have used \eqref{eq:bs_inq_1} for the second inequality sign.
We next use \eqref{eq:bs_inq_2} to estimate $\|\bs\Pi_k\mb q-\mb q\|_{\pp K}$. This completes the proof.

\end{proof}

We are now ready to prove Theorem \ref{thm:hdgp_pj_ell}.
By Proposition \ref{prop:g+}, we know  $\bs\Pi^\mr{H+}$ and $\delta_{\pm\tau}^{\Pi^\mr{H+}}$ are well defined.
We next prove that $\bs\Pi^\mr{H+}$ satisfies equations \eqref{eq:hdgp_pj_ell_eqn}.
Equation \eqref{eq:hdgp_pj_ell_eqn_1} holds obviously since $\Pi^\mr{H+} u=\Pi_{k+1}u$. Equation \eqref{eq:hdgp_pj_ell_eqn_2} holds by the definition \eqref{eq:def_hdg+_ell_4}.
To prove \eqref{eq:hdgp_pj_ell_eqn_3}, first note that
\begin{align}
\label{eq:cv_p_9}
&(\div(\bs\Pi^\mr{H+}\mb q-\mb q),w)_K
\pm\dualpr{\tau\mr P_M(\Pi_{k+1}u-u),w}_{\pp K}\\
\nonumber
&\qquad=\dualpr{(\bs\Pi^\mr{H+}\mb q-\mb q)\cdot\mb n\pm\tau\mr P_M(\Pi_{k+1}u-u),w}_{\pp K}
-(\bs\Pi^\mr{H+}\mb q-\mb q,\nabla w)_K,
\end{align}
for all $w\in\mc P_{k+1}(K)$.
Now \eqref{eq:hdgp_pj_ell_eqn_3} follows by using \eqref{eq:cv_p_9} and \eqref{eq:def_hdg+_ell_3}.

We next prove \eqref{eq:hdgp_pj_ell_eqn_conv}.
By \eqref{eq:bs_inq_2} and \eqref{eq:g+_conv}, we know that $\bs\Pi^\mr{H+}\mb q$ and $\Pi^\mr{H+} u=\Pi_{k+1}u$ converge optimally. It only remains to estimate the boundary remainder. By the definition \eqref{eq:def_hdg+_ell_4} and the fact that $\|\tau\|_{L^\infty(\pp K)}\le c_2h_K^{-1}$, we have
\begin{align*}
\|\delta_{\pm\tau}^{\Pi^\mr{H+}}(\mb q,u)\|_{\pp K}\le \|\mr P_M(\bs\Pi^\mr{H+}\mb q\cdot\mb n-\mb q\cdot\mb n)\|_{\pp K}+c_2h_K^{-1}\|\mr P_M(\Pi_{k+1}u-u)\|_{\pp K}.
\end{align*}
By \eqref{eq:bs_inq_2} and \eqref{eq:g+_conv} again, we complete the proof.

\subsection{Proof of Theorem \ref{thm:est_ell}}\label{sec:prf_est}
In this subsection, we give a step-by-step proof for Theorem \ref{thm:est_ell}. The proof will be very similar to those used in \cite{CoGoSa:2010}, thanks to the introduction of the HDG+ projection. In this way, we are able to reuse the existing projection-based error analysis for the analysis of the HDG+ method.

\quad\\
\noindent
{\bf Step 1: Error equations.} We first define the error terms:
\begin{alignat*}{5}
\bs\varepsilon_h^q:=\bs\Pi^\mr{H+}\mb q-\mb q_h\in\mb V_h,\quad \varepsilon_h^u:=\Pi^\mr{H+} u-u_h\in W_h,\quad \widehat{\varepsilon}_h^{\,u}:=\mr P_Mu-\widehat{u}_h\in M_h. 
\end{alignat*}
Now, by testing \eqref{eq:ell_pde} with $(\mb r,w,\mu)\in\mb V_h\times W_h\times M_h$ and then using \eqref{eq:hdgp_pj_ell_eqn},
we obtain the projection equations:
\begin{subequations}\label{eq:prj_eqns}
\begin{align}
\label{eq:prj_eqns_1}
(\kappa^{-1}\bs\Pi^\mr{H+}\mb q,\mb r)_{\mc T_h} - (\Pi^\mr{H+} u,\div\mb r)_{\mc T_h} + \dualpr{\mr P_Mu,\mb r\cdot\mb n}_{\pp\mc T_h} 
&= (\kappa^{-1}(\bs\Pi^\mr{H+}\mb q-\mb q),\mb r)_{\mc T_h},\\
\label{eq:prj_eqns_2}
(\div\bs\Pi^\mr{H+}\mb q,w)_{\mc T_h} + \dualpr{\tau\mr P_M(\Pi^\mr{H+} u-\mr P_Mu),w}_{\pp\mc T_h}
&=(f,w)_{\mc T_h}\\
\nonumber
&\quad\, +\dualpr{\delta_\tau^{\Pi^\mr{H+}}(\mb q,u), w}_{\pp\mc T_h},\\
\label{eq:prj_eqns_3}
-\dualpr{\bs\Pi^\mr{H+}\mb q\cdot\mb n+\tau(\Pi^\mr{H+} u-u),\mu}_{\pp\mc T_h\backslash\Gamma} &= -\dualpr{\delta_\tau^{\Pi^\mr{H+}}(\mb q,u),\mu}_{\pp\mc T_h\backslash\Gamma},\\
\label{eq:prj_eqns_4}
\dualpr{\mr P_Mu,\mu}_{\Gamma} &= \dualpr{g,\mu}_\Gamma,
\end{align}
for all $(\mb r,w,\mu)\in\mb V_h\times W_h\times M_h$.
\end{subequations}
In the above equations, \eqref{eq:prj_eqns_1}, \eqref{eq:prj_eqns_2}, and \eqref{eq:prj_eqns_3} are obtained by using \eqref{eq:hdgp_pj_ell_eqn_1}, \eqref{eq:hdgp_pj_ell_eqn_2}, and \eqref{eq:hdgp_pj_ell_eqn_3}, respectively. The equation \eqref{eq:prj_eqns_4} holds obviously since $\mr P_M\big|_{\pp K}$ is the $L^2$ projection to $\mc R_k(\pp K)$ for all $K\in\mc T_h$.

By taking the difference between \eqref{eq:prj_eqns} and \eqref{eq:HDG_s}, we obtain the error equations:
\begin{subequations}\label{eq:err_eqns}
\begin{align}
\label{eq:err_eqns_1}
(\kappa^{-1}\bs\varepsilon_h^q,\mb r)_{\mc T_h} - (\varepsilon_h^u,\div\mb r)_{\mc T_h} + \dualpr{\widehat{\varepsilon}_h^{\,u},\mb r\cdot\mb n}_{\pp\mc T_h} &= (\kappa^{-1}(\bs\Pi^\mr{H+}\mb q-\mb q),\mb r)_{\mc T_h},\\
\label{eq:err_eqns_2}
(\div\bs\varepsilon_h^q,w)_{\mc T_h} + \dualpr{\tau\mr P_M(\varepsilon_h^u-\widehat{\varepsilon}_h^{\,u}),w}_{\pp\mc T_h}
&=\dualpr{\delta_\tau(\mb q,u), w}_{\pp\mc T_h},\\
\label{eq:err_eqns_3}
-\dualpr{\bs\varepsilon_h^q\cdot\mb n+\tau(\varepsilon_h^u-\widehat{\varepsilon}_h^{\,u}),\mu}_{\pp\mc T_h\backslash\Gamma} &= -\dualpr{\delta_\tau(\mb q,u),\mu}_{\pp\mc T_h\backslash\Gamma},\\
\label{eq:err_eqns_4}
\dualpr{\widehat{\varepsilon}_h^{\,u},\mu}_{\Gamma} &= 0,
\end{align}
for all $(\mb r,w,\mu)\in\mb V_h\times W_h\times M_h$.
\end{subequations}

\quad\\
\noindent{\bf Step 2: Energy identity.}
By testing the error equations
with $\mb r=\bs\varepsilon_h^q$, $w=\varepsilon_h^u$, $\mu=\widehat{\varepsilon}_h^{\,u}$ in \eqref{eq:err_eqns_1}-\eqref{eq:err_eqns_3} and adding the equations, then using \eqref{eq:err_eqns_4}, which suggests that $\widehat{\varepsilon}_h^{\,u}\big|_\Gamma=0$, we obtain the following energy identity:
\begin{align}
\label{eq:ene_id}
&(\kappa^{-1}\bs\varepsilon_h^q,\bs\varepsilon_h^q)_{\mc T_h}+\dualpr{\tau\mr P_M(\varepsilon_h^u-\widehat{\varepsilon}_h^{\,u}),\varepsilon_h^u-\widehat{\varepsilon}_h^{\,u}}_{\pp\mc T_h}\\
\nonumber
&\qquad\qquad=(\kappa^{-1}(\bs\Pi^\mr{H+}\mb q-\mb q),\bs\varepsilon_h^q)_{\mc T_h}+\dualpr{\delta_\tau^{\Pi^\mr{H+}}(\mb q,u),\varepsilon_h^u-\widehat{\varepsilon}_h^{\,u}}_{\pp\mc T_h}.
\end{align}
By using the energy identity \eqref{eq:ene_id}, we easily obtain 
\begin{align}\label{eq:est_qh_jump}
\|\kappa^{-\frac{1}{2}}\bs\varepsilon_h^q\|_{\mc T_h}^2+\|\tau^{\frac{1}{2}}\mr P_M(\varepsilon_h^u-\widehat{\varepsilon}_h^{\,u})\|_{\pp\mc T_h}^2
\le \|\kappa^{-\frac{1}{2}}(\bs\Pi^\mr{H+}\mb q-\mb q)\|_{\mc T_h}^2
+\|\tau^{-\frac{1}{2}}\delta_\tau^{\Pi^\mr{H+}}(\mb q,u)\|_{\pp\mc T_h}^2.
\end{align}
This proves \eqref{eq:est_qh}. We are next going to prove \eqref{eq:est_uh} and \eqref{eq:est_uhhat}.

\quad\\
\noindent{\bf Step 3: Duality identity.} 
We first introduce the duality equations of \eqref{eq:ell_pde}:
\begin{subequations}\label{eq:dual_pde}
\begin{alignat}{5}
\kappa^{-1}\bs\psi - \nabla \phi &= \mb 0 &\qquad& \mr{in}\ \Omega,\\
-\div\bs\psi & = \theta&&\mr{in}\ \Omega,\\
\phi & =0 &&\mr{on}\ \Gamma,
\end{alignat}
\end{subequations}
We next define the projections and the boundary remainder of the solutions of the duality equations \eqref{eq:dual_pde}:
\begin{alignat*}{5}
\bs\Pi^\mr{H+}\bs\psi:=\prod_{K\in\mc T_h}\bs\Pi^\mr{H+}\bs\psi,\quad
\Pi^\mr{H+}\phi:=\prod_{K\in\mc T_h}\Pi^\mr{H+}\phi,\quad
\delta_{-\tau}^{\Pi^\mr{H+}}(\bs\psi,\phi):=\prod_{K\in\mc T_h}\delta_{-\tau}^{\Pi^\mr{H+}}(\bs\psi,\phi).
\end{alignat*}
Note that we used $-\tau$ to define the boundary remainder.
By testing \eqref{eq:dual_pde} with $(\mb r,w,\mu)\in\mb V_h\times W_h\times M_h$ and then using \eqref{eq:hdgp_pj_ell_eqn},
we obtain the following equations in a similar way we obtained \eqref{eq:prj_eqns}:
\begin{subequations}\label{eq:dual_pj_eqns}
\begin{align}
\label{eq:dual_pj_eqns_1}
(\kappa^{-1}\bs\Pi^\mr{H+}\bs\psi,\mb r)_{\mc T_h} + (\Pi^\mr{H+}\phi,\div\mb r)_{\mc T_h} - \dualpr{\mr P_M\phi,\mb r\cdot\mb n}_{\pp\mc T_h} &= (\kappa^{-1}(\bs\Pi^\mr{H+}\bs\psi-\bs\psi),\mb r)_{\mc T_h},\\
\label{eq:dual_pj_eqns_2}
-(\div\bs\Pi^\mr{H+}\bs\psi,w)_{\mc T_h} + \dualpr{\tau\mr P_M(\Pi^\mr{H+}\phi-\mr P_M\phi),w}_{\pp\mc T_h}
&=(\theta,w)_{\mc T_h}\\
\nonumber
&\quad\, -\dualpr{\delta_{-\tau}^{\Pi^\mr{H+}}(\bs\psi,\phi), w}_{\pp\mc T_h},\\
\label{eq:dual_pj_eqns_3}
\dualpr{\bs\Pi^\mr{H+}\bs\psi\cdot\mb n-\tau(\Pi^\mr{H+}\phi-\phi),\mu}_{\pp\mc T_h\backslash\Gamma} &= \dualpr{\delta_{-\tau}^{\Pi^\mr{H+}}(\bs\psi,\phi),\mu}_{\pp\mc T_h\backslash\Gamma},\\
\label{eq:dual_pj_eqns_4}
\dualpr{\mr P_M\phi,\mu}_{\Gamma} &= 0,
\end{align}
\end{subequations}
for all $(\mb r,w,\mu)\in\mb V_h\times W_h\times M_h$. 
Now we test \eqref{eq:err_eqns_1}-\eqref{eq:err_eqns_3} with 
$\mb r=\bs\Pi^\mr{H+}\psi$, $w=\Pi^\mr{H+}\phi$, $\mu=\mr P_M\phi$, test \eqref{eq:dual_pj_eqns_1}-\eqref{eq:dual_pj_eqns_3} with $\mb r=\bs\varepsilon_h^q$, $w=\varepsilon_h^u$, $\mu=\widehat{\varepsilon}_h^{\,u}$, and use \eqref{eq:err_eqns_4} and \eqref{eq:dual_pj_eqns_4}, which imply $\widehat{\varepsilon}_h^{\,u}\big|_\Gamma=\mr P_M\phi\big|_\Gamma=0$. 
Comparing the two sets of equations, we obtain
\begin{align*}
&(\kappa^{-1}(\bs\Pi^\mr{H+}\mb q-\mb q),\bs\Pi^\mr{H+}\bs\psi)_{\mc T_h}+\dualpr{\delta_\tau^{\Pi^\mr{H+}}(\mb q,u),\Pi^\mr{H+}\phi-\mr P_M\phi}_{\pp\mc T_h}\\
&\qquad =(\kappa^{-1}(\bs\Pi^\mr{H+}\bs\psi-\bs\psi),\bs\varepsilon_h^q)_{\mc T_h}+(\theta,\varepsilon_h^u)_{\mc T_h}
-\dualpr{\delta_{-\tau}^{\Pi^\mr{H+}}(\bs\psi,\phi),\varepsilon_h^u-\widehat{\varepsilon}_h^{\,u}}_{\pp\mc T_h}.
\end{align*}
Rearranging the terms of the above identity, we have the following duality identity:
\begin{align}
\label{eq:dual_id}
(\theta,\varepsilon_h^u)_{\mc T_h}
&=(\bs\Pi^\mr{H+}\mb q-\mb q,\nabla\phi)_{\mc T_h}
+(\kappa^{-1}(\bs\Pi^\mr{H+}\bs\psi-\bs\psi),\mb q_h-\mb q)_{\mc T_h}\\
\nonumber
&\quad +\dualpr{\delta_\tau^{\Pi^\mr{H+}}(\mb q,u),\mr P_M\Pi^\mr{H+} \phi-\mr P_M\phi}_{\pp\mc T_h}
+\dualpr{\delta_{-\tau}^{\Pi^\mr{H+}}(\bs\psi,\phi),\mr P_M\varepsilon_h^u-\widehat{\varepsilon}_h^{\,u}}_{\pp\mc T_h}.
\end{align}

\quad\\
\noindent{\bf Step 4: Estimating $u_h$ and $\widehat{u}_h$.} 
We first consider the case when $k\ge1$. Then $(\bs\Pi^\mr{H+}\mb q-\mb q,\bs\Pi_0\nabla\phi)_{\mc T_h}=0$ because of \eqref{eq:def_hdg+_ell_3} (taking $w\in\mc P_1(K)$) and the assumption $k\ge1$.
By \eqref{eq:hdgp_pj_ell_eqn_conv} with $m_1=r_0$ and $m_2=1+r_0$, \eqref{eq:bs_inq_2} with $s=r_0$, and then using \eqref{eq:ell_reg}, we have
\begin{align*}
&\|\nabla\phi-\bs\Pi_0\nabla\phi\|_{\mc T_h}+\|\bs\Pi^\mr{H+}\bs\psi-\bs\psi\|_{\mc T_h}\\
&+\|\tau^{\frac{1}{2}}(\mr P_M\Pi^\mr{H+}\phi-\mr P_M\phi)\|_{\pp\mc T_h}+\|\tau^{-\frac{1}{2}}\delta_{-\tau}^{\Pi^\mr{H+}}(\bs\psi,\phi)\|_{\pp\mc T_h}\\
&\qquad\qquad\lesssim h^{r_0}(|\bs\psi|_{r_0,\Omega}+|\phi|_{1+r_0,\Omega})\lesssim h^{r_0}\|\theta\|_{\mc T_h}.
\end{align*}
Taking $\theta=\varepsilon_h^u$ in \eqref{eq:dual_id}, we have
\begin{align*}
\|\varepsilon_h^u\|_{\mc T_h}\lesssim h^{r_0}\big(&\|\bs\Pi^\mr{H+}\mb q-\mb q\|_{\mc T_h}+\|\mb q_h-\mb q\|_{\mc T_h}\\
&+\|\tau^{-\frac{1}{2}}\delta_\tau^{\Pi^\mr{H+}}(\mb q,u)\|_{\pp\mc T_h}+\|\tau^{\frac{1}{2}}(\mr P_M\varepsilon_h^u-\widehat{\varepsilon}_h^{\,u})\|_{\pp\mc T_h}\big).
\end{align*}
The above inequality with \eqref{eq:est_qh_jump} implies \eqref{eq:est_uh}.

We now consider the case when $k=0$. The only term we need to take into special consideration is $(\bs\Pi^\mr{H+}\mb q-\mb q,\nabla\phi)_{\mc T_h}$. We first rewrite it as follows:
\begin{align*}
(\bs\Pi^\mr{H+}\mb q-\mb q,\nabla\phi)_{\mc T_h}
=(\bs\Pi^\mr{H+}\mb q-\mb q,\nabla\phi-\nabla(\Pi_1\phi))_{\mc T_h}
+(\bs\Pi^\mr{H+}\mb q-\mb q,\nabla(\Pi_1\phi))_{\mc T_h}.
\end{align*}
By \eqref{eq:bs_inq_2} with $m=1$ and $s=2$, the first term of the above equation can be handled similarly as in the case $k\ge1$. We next focus on the second term. By \eqref{eq:def_hdg+_ell_3}, we have
\begin{align*}
(\bs\Pi^\mr{H+}\mb q-\mb q,\nabla(\Pi_1\phi))_{\mc T_h}
=\dualpr{\mr P_M(\mb q\cdot\mb n)-\mb q\cdot\mb n,\Pi_1\phi}_{\pp\mc T_h},
\end{align*}
where again, $\mr P_M$ is the $L^2$ projection to $\prod_{K\in\mc T_h}\mc R_0(\pp K)$. Let $\mb P_0^F$ be the $L^2$ projection to $\mc P_0(F)^d$, and $\mc E_h^\Gamma$ and $\mc E_h^\circ$ be the collections of the boundary and the interior faces of $\mc E_h$, respectively. Then we have
\begin{align*}
\dualpr{\mr P_M(\mb q\cdot\mb n)-\mb q\cdot\mb n,\Pi_1\phi}_{\pp\mc T_h}
&=\sum_{F\in\mc E_h^\Gamma}\dualpr{\mb P_0^F\mb q-\mb q,(\Pi_1\phi-\phi)\mb n_F}_{F}\\
&\quad+\sum_{F\in\mc E_h^\circ}\dualpr{\mb P_0^F\mb q-\mb q,\Pi_1\phi(\mb n_F^++\mb n_F^-)}_{F},
\end{align*}
where we have used the fact that $\phi\big|_\Gamma=0$ and $\mr P_M(\mb q\cdot\mb n)\big|_F = (\mb P_0^F\mb q)\cdot\mb n$ (noticing $\dualpr{\mr P_M(\mb q\cdot\mb n),\mu}_{F}=\dualpr{\mb q,\mu\mb n}_{F}=\dualpr{\mb P_0^F\mb q,\mu\mb n}_{F}=\dualpr{\mb P_0^F\mb q\cdot\mb n,\mu}_{F}$ for all $\mu\in\mc P_0(F)$). Hence
\begin{align*}
|\dualpr{\mr P_M(\mb q\cdot\mb n)-\mb q\cdot\mb n,\Pi_1\phi}_{\pp\mc T_h}|
&\le 2 \sum_{F\in\mc E_h}\|\mb P_0^F\mb q-\mb q\|_F\|\Pi_1\phi-\phi\|_F\\
&\le 2 \sum_{K\in\mc T_h}\|\bs\Pi_0\mb q-\mb q\|_{\pp K}
\|\Pi_1\phi-\phi\|_{\pp K}\\
&\lesssim h^{r_0}\|h_K^{\frac{1}{2}}(\bs\Pi_0\mb q-\mb q)\|_{\pp\mc T_h}|\phi|_{1+r_0,\Omega}.
\end{align*}
The rest is similar to the case when $k\ge1$.

It now only remains to estimate the term $\|\mr P_Mu-\widehat{u}_h\|_h$. 
First note that 
\begin{align}
\label{eq:prf_15}
\|\widehat{\varepsilon}_h^{\,u}\|_h^2=\sum_{K\in\mc T_h}h_K\|\widehat{\varepsilon}_h^{\,u}\|_{\pp K}^2
\approx\sum_{K\in\mc T_h}h_K^2\|\tau^{\frac{1}{2}}\widehat{\varepsilon}_h^{\,u}\|_{\pp K}^2
\le h^2\|\tau^{\frac{1}{2}}\widehat{\varepsilon}_h^{\,u}\|_{\pp\mc T_h}^2.
\end{align}
By \eqref{eq:est_qh_jump}, we have
\begin{align}\label{eq:prf_11}
\|\tau^{\frac{1}{2}}\widehat{\varepsilon}_h^{\,u}\|_{\pp\mc T_h}\lesssim \|\tau^{\frac{1}{2}}\mr P_M\varepsilon_h^u\|_{\pp\mc T_h}+\|\bs\Pi^\mr{H+}\mb q-\mb q\|_{\mc T_h}+\|\tau^{-\frac{1}{2}}\delta_\tau^{\Pi^\mr{H+}}(\mb q,u)\|_{\pp\mc T_h}.
\end{align}
By using \eqref{eq:est_uh}, we can estimate the term $\|\tau^{\frac{1}{2}}\mr P_M\varepsilon_h^u\|_{\pp\mc T_h}$ as follows:
\begin{align}
\label{eq:prf_13}
\|\tau^{\frac{1}{2}}\mr P_M\varepsilon_h^u\|_{\pp\mc T_h}^2
&=\sum_{K\in\mc T_h}\|\tau^{\frac{1}{2}}\mr P_M(\Pi^\mr{H+} u-u_h)\|_{\pp K}^2
\lesssim \sum_{K\in\mc T_h}h_K^{-2}\|\Pi^\mr{H+} u-u_h\|_{K}^2\\
\nonumber
&\lesssim h_\mr{min}^{-2}h^{2r_0}\left(
\|\bs\Pi^\mr{H+}\mb q-\mb q\|_{\mc T_h}
+\|\tau^{-\frac{1}{2}}\delta_\tau^{\Pi^\mr{H+}}(\mb q,u)\|_{\pp\mc T_h}
\right)^2.
\end{align}
Combining \eqref{eq:prf_15}, \eqref{eq:prf_11}, and \eqref{eq:prf_13}, we obtain \eqref{eq:est_uhhat}. This completes the proof.

\section{The projection for elasticity}
\label{sec:hdgp_pj_elas}
\subsection{Main results}
In \cite{DuSa_elas:2020}, we devised the HDG+ projection for elasticity on simplicial elements.
In this section, we extend the projection (see \cite[Theorem 2.1]{DuSa_elas:2020}) to polyhedral elements. This new projection will render all the analysis and estimates in \cite[Sections 5,6\&7]{DuSa_elas:2020} valid for general polyhedral meshes. (The three sections in \cite{DuSa_elas:2020} cover the error analysis of the HDG+ methods for steady-state elasticity, time-harmonic elastodynamics, and transient elastic waves, respectively.)

\begin{subequations}\label{eq:def_hdg+_elas}
For each $K\in\mc T_h$, we define the HDG+ projection for elasticity as follows:
\begin{align*}
\bs\Pi^\mr{H+}: H^{\frac{1}{2}+\epsilon}(K;\mbb R_\mr{sym}^{d\times d})
\times H^{\frac{1}{2}+\epsilon}(K;\mbb R^{d})
&\rightarrow
\mc P_k(K;\mbb R_\mr{sym}^{d\times d})\times\mc P_{k+1}(K;\mbb R^d),\\
(\mb q,u)&\mapsto (\bs\Pi^\mr{H+}\bs\sigma,\bs\Pi^\mr{H+}\mb u),
\end{align*}
where the first component $\bs\Pi^\mr{H+}$ is defined by solving
\begin{alignat}{5}
\label{eq:def_hdg+_elas_2}
(\bs\Pi^\mr{H+}\bs\sigma-\bs\sigma,\bs\theta)_K &=0 &\quad& \forall\bs\theta\in\bs\varepsilon(\mc P_{k+1}(K;\mbb R^d))^{\perp_k},\\
\label{eq:def_hdg+_elas_3}
(\bs\Pi^\mr{H+}\bs\sigma-\bs\sigma,\bs\varepsilon(\bs v))_K&=\dualpr{\bs{\mr P}_M(\bs\sigma\mb n)-\bs\sigma\mb n,\bs v}_{\pp K}
&& \forall \bs v\in \mc P_{k+1}(K;\mbb R^d),
\end{alignat}
and the second component $\bs\Pi^\mr{H+}\mb u:=\bs\Pi_{k+1}\mb u$ as the $L^2$ projection to $\mc P_{k+1}(K;\mbb R^d)$.
In the above equations, $(\cdot)^{\perp_k}$ represents the orthogonal complement in the background space $\mc P_k(K;\mbb R_\mr{sym}^{d\times d})$, the notation $\bs\varepsilon(\bs v):=\frac{1}{2}(\nabla\bs v+(\nabla\bs v)^\perp)$ represents the symmetric gradient, and
$\bs{\mr P}_M:L^2(\pp K;\mbb R^d)\rightarrow\mc R_k(\pp K;\mbb R^d)$ is the $L^2$ projection to the range space.

We define the associated boundary remainder as follows:
\begin{align}\label{eq:def_hdg+_elas_4}
\bs\delta_{\pm\tau}^{\Pi^\mr{H+}}(\bs\sigma,\bs u):= -(\bs\Pi^\mr{H+}\bs\sigma\,\mb n-\bs{\mr P}_M(\bs\sigma\mb n))\pm\bs\tau(\bs{\mr P}_M\bs\Pi^\mr{H+}\bs u-\bs{\mr P}_M\bs u)\in\mc R_k(\pp K;\mbb R^d),
\end{align}
\end{subequations}
where $\bs\tau\in\mc R_0(\pp K;\mbb R_\mr{sym}^{d\times d})$ satisfying \cite[Eqn. (2.1)]{DuSa_elas:2020}, namely, $\bs\tau$ is uniformly bounded and coercive.

The main result in this section is the following theorem. For notational convenience, we hide the dependence of $\bs\delta_{\pm\tau}^{\Pi^\mr{H+}}$ on $(\bs\sigma,\bs u)$.

\begin{theorem}[HDG+ projection for elasticity]\label{thm:pj_elas}
The projection $\bs\Pi^\mr{H+}$ and the remainder $\bs\delta_{\pm\tau}^{\Pi^\mr{H+}}$ are well defined 
by \eqref{eq:def_hdg+_elas} and they satisfy
\begin{subequations}\label{eq:hdgp_pj_el}
\begin{alignat}{5}
\label{eq:hdgp_pj_el_1}
(\bs\Pi^\mr{H+}\bs u -\bs u,\bs v)_K &= 0 \\
\nonumber&\quad\forall\bs v\in\mc P_{k-1}(K;\mbb R^d),\\
\label{eq:hdgp_pj_el_2}
\dualpr{-(\bs\Pi^\mr{H+}\bs\sigma\,\mb n-\bs\sigma\mb n)\pm\bs\tau(\bs\Pi^\mr{H+}\bs u-\bs u),\bs\mu}_{\pp K} &= 
\dualpr{\bs\delta_{\pm\tau}^{\Pi^\mr{H+}},\bs\mu}_{\pp K}\\
\nonumber&\quad\forall \bs\mu\in\mc R_k(\pp K;\mbb R^d),\\
\label{eq:hdgp_pj_el_3}
-(\div(\bs\Pi^\mr{H+}\bs\sigma-\bs\sigma),\bs w)_K\pm\dualpr{\bs\tau\bs{\mr P}_M(\bs\Pi^\mr{H+}\bs u-\bs u),\bs w}_{\pp K}
&= \dualpr{\bs\delta_{\pm\tau}^{\Pi^\mr{H+}},\bs w}_{\pp K}\\
\nonumber&\quad\forall \bs w\in\mc P_{k+1}(K;\mbb R^d).
\end{alignat}
\end{subequations}
Furthermore,
\begin{align}\label{eq:hdgp_pj_el_conv}
\|\bs\Pi^\mr{H+}\bs\sigma-\bs\sigma\|_K+h_K^{-1}\|\bs\Pi^\mr{H+}\bs u-\bs u\|_K + h_K^{\frac{1}{2}}\|\bs\delta_{\pm\tau}^{\Pi^\mr{H+}}\|_{\pp K}
\le Ch_K^m(|\bs\sigma|_{m,K}+|\bs u|_{m+1,K}),
\end{align}
where $m\in[\frac{1}{2}+\epsilon,k+1]$. Here, the constant $C$ depends only on $k$, $\gamma_K$, and the upper bound of $\bs\tau$.
\end{theorem}

Note that the two boundary remainders $\bs\delta_{+\tau}^{\Pi^\mr{H+}}$ and $\bs\delta_{-\tau}^{\Pi^\mr{H+}}$ correspond to the HDG+ projection and the adjoint projection in \cite[Theorem 2.1]{DuSa_elas:2020}, respectively.
We also remark that we have used the HDG+ projection to define the initial velocity for the semi-discrete HDG+ scheme in \cite{DuSa_elas:2020}. Therefore,  equations \eqref{eq:def_hdg+_elas} provide a way of calculating the initial conditions for the semi-discrete scheme for elastic waves.

\subsection{Proof of Theorem \ref{thm:pj_elas}}
In this subsection, we prove Theorem \ref{thm:pj_elas}. The proof here will be similar to the proof of Theorem \ref{thm:hdgp_pj_ell} in Section \ref{sec:ell_hdg+_pfs}. 

\begin{prop}\label{prop:sg+}
The projection $\bs\Pi^\mr{H+}$ is well defined by \eqref{eq:def_hdg+_elas_2} and \eqref{eq:def_hdg+_elas_3}, and we have
\begin{align}\label{eq:sg+_conv}
h_K^{\frac{1}{2}}\|\bs\Pi^\mr{H+}\bs\sigma-\bs\sigma\|_{\pp K}+\|\bs\Pi^\mr{H+}\bs\sigma-\bs\sigma\|_K\le 
Ch_K^m|\bs\sigma|_{m,K},
\end{align}
where $m\in[\frac{1}{2}+\epsilon,k+1]$. Here, the constant $C$ depends only on $k$ and the shape-regularity constant $\gamma_K$.
\end{prop}
\begin{proof}
We can easily verify that \eqref{eq:def_hdg+_elas_2} and \eqref{eq:def_hdg+_elas_3} define a square system. We next prove the convergence equation \eqref{eq:sg+_conv}, from which the unique solvability of \eqref{eq:def_hdg+_elas_2} and \eqref{eq:def_hdg+_elas_3} follows automatically. Let $\bs\varepsilon_\sigma:=\bs\Pi^\mr{H+}\bs\sigma-\bs\Pi_k\bs\sigma$. By \eqref{eq:def_hdg+_elas_2} and \eqref{eq:def_hdg+_elas_3}, we have
\begin{subequations}\label{eq:strpj_pf_1}
\begin{alignat}{5}
\label{eq:strpj_pf_1a}
(\bs\varepsilon_\sigma,\bs\theta)_K &=0 &\quad& \forall\bs\theta\in \bs\varepsilon(\mc P_{k+1}(K;\mbb R^d))^{\perp_k},\\
\label{eq:strpj_pf_1b}
(\bs\varepsilon_\sigma,\bs\varepsilon(\bs v))_K&=\dualpr{\bs{\mr P}_M(\bs\sigma\mb n)-\bs\sigma\mb n,\bs v}_{\pp K}
&&\forall \bs v\in \mc P_{k+1}(K;\mbb R^d).
\end{alignat}
\end{subequations}
We now decompose $\bs\varepsilon_\sigma$ into the summation $\bs\varepsilon_\sigma=\bs\varepsilon_\sigma^1+\varepsilon_\sigma^2$, where $\varepsilon_\sigma^1\in \bs\varepsilon(\mc P_{k+1}(K;\mbb R^d))$ and $\varepsilon_\sigma^2\in \bs\varepsilon(\mc P_{k+1}(K;\mbb R^d))^{\perp_k}$. Since $\bs\varepsilon_\sigma^1\in\bs\varepsilon(\mc P_{k+1}(K;\mbb R^d))$, we can write $\bs\varepsilon_\sigma^1=\bs\varepsilon(\bs p+\bs m)$ for some $\bs p\in\mc P_{k+1}(K;\mbb R^d)$ and arbitrary rigid motion $\bs m\in\mc M$.  By \eqref{eq:strpj_pf_1a} and \eqref{eq:strpj_pf_1b} we have
\begin{align*}
\|\bs\varepsilon_\sigma\|_{K}^2=(\bs\varepsilon_\sigma,\bs\varepsilon_\sigma^1)_K
=(\bs\varepsilon_\sigma,\bs\varepsilon(\bs p+\bs m))_K
=\dualpr{\bs{\mr P}_M(\bs\sigma\mb n)-\bs\sigma\mb n,\bs p+\bs m}_{\pp K}.
\end{align*}
We next apply \cite[Lemma 4.1]{QiShSh:2018} to the term $\bs p+\bs m$ and then obtain
\begin{align*}
\|\bs\varepsilon_\sigma\|_{K}^2\lesssim h_K^{\frac{1}{2}}\|\bs{\mr P}_M(\bs\sigma\mb n)-\bs\sigma\mb n\|_{\pp K}\|\bs\varepsilon(\bs p)\|_K\le h_K^{\frac{1}{2}}\|\bs{\mr P}_M(\bs\sigma\mb n)-\bs\sigma\mb n\|_{\pp K}\|\bs\varepsilon_\sigma\|_K.
\end{align*}
To estimate the boundary term $\|\bs\Pi^\mr{H+}\bs\sigma-\bs\sigma\|_{\pp K}$, note that
\begin{align*}
\|\bs\Pi^\mr{H+}\bs\sigma-\bs\sigma\|_{\pp K}\le \|\bs\Pi_k\bs\sigma-\bs\sigma\|_{\pp K}
+ \|\bs\varepsilon_\sigma\|_{\pp K}
\lesssim \|\bs\Pi_k\bs\sigma-\bs\sigma\|_{\pp K}
+ h_K^{-\frac{1}{2}}\|\bs\varepsilon_\sigma\|_{K}.
\end{align*}
We then use \eqref{eq:bs_inq_2} and the proof is completed.

\end{proof}

Let us now prove Theorem \ref{thm:pj_elas}.
By Proposition \ref{prop:sg+}, we know $\bs\Pi^\mr{H+}$ and $\bs\delta_{\pm\tau}^{\Pi^\mr{H+}}$ are well defined by \eqref{eq:def_hdg+_elas}.
We next prove equations \eqref{eq:hdgp_pj_el}.
Equations \eqref{eq:hdgp_pj_el_1} and \eqref{eq:hdgp_pj_el_2} hold obviously by the definition $\bs\Pi^\mr{H+}\mb u:=\bs\Pi_{k+1}\mb u$ and \eqref{eq:def_hdg+_elas_4}.

To prove \eqref{eq:hdgp_pj_el_3}, first note that
\begin{align}
\label{eq:pf_17}
&-(\div(\bs\Pi^\mr{H+}\bs\sigma-\bs\sigma),\bs w)_K
\pm\dualpr{\bs\tau\bs{\mr P}_M(\bs\Pi^\mr{H+}\bs u-\bs u),\bs w}_{\pp K}\\
\nonumber
&\qquad=\dualpr{-(\bs\Pi^\mr{H+}\bs\sigma\,\mb n-\bs\sigma\mb n)\pm\bs\tau\bs{\mr P}_M(\bs\Pi^\mr{H+}\bs u-\bs u),\bs w}_{\pp K}+(\bs\Pi^\mr{H+}\bs\sigma-\bs\sigma,\bs\varepsilon(\bs w))_K,
\end{align}
for all $\bs w\in\mc P_{k+1}(K;\mbb R^d)$. 
Equations \eqref{eq:pf_17} and \eqref{eq:def_hdg+_elas_3} then imply \eqref{eq:hdgp_pj_el_3}. 

The convergence property \eqref{eq:hdgp_pj_el_conv} holds because of equations \eqref{eq:sg+_conv} and \eqref{eq:bs_inq_2}, and the fact that $\bs\tau$ is uniformly bounded. This completes the proof.

\section*{Conclusions}
We have devised two new HDG+ projections on polyhedral elements, extending our previous results in \cite{DuSa:2019} for elliptic problems and the results in \cite{DuSa_elas:2020} for elasticity to polyhedral meshes. The projections here are constructed in a different way without using the $M$-decompositions as a middle step. Consequently, the construction is more straightforward. 

We would like to mention that in \cite{CoDiEr:2016} and also in \cite[Section 5.1.6]{DiDr:2020}, connections between Hybrid-High order (HHO) methods and HDG methods are established, making possible to ``incorporate into HDG methods the new, subtle way of defining the numerical trace for the flux in HHO methods" (see \cite[Conclusion]{CoDiEr:2016}). As a result, the HDG+ method can be associated to one of these HDG methods with HHO stabilization and analyzed within the HHO framework. 
Though, benefits of analyzing the HDG+ method with projection-based analysis includes: (1) reusing existing analysis techniques of HDG methods, such as those in \cite{DuSa_elas:2020}, where the HDG+ projection was used to devise and analyze a semi-discrete HDG+ method for transient elastic waves, by using existing analysis techniques from \cite{CoFuHuJiSaSa:2018}, where a standard HDG method for acoustic waves was proposed and analyzed by the classical HDG projection; (2) a simple and concise analysis of mixed-type methods where different stabilizations and approximation spaces are used on different elements, by adopting correspondingly different projections to capture the features on each element to minimize the boundary remainder and then obtain a single form of the energy/duality identity. 
A natural question is whether the projection-remainder way of analysis established in this paper can be applied those HDG methods with more subtle stabilization functions. 
This constitutes a possible future work.

\vspace{0.5cm}
\noindent{\bf Acknowledgments.} This work was partially supported by the NSF grant DMS-1818867. The first author would like to thank B. Cockburn for the discussions that lead the paper to a better form, and for supporting his visit to the University of Minnesota, where the paper was written. The author also wants to thank the anonymous referees for their revision suggestions which have improved the presentation of the paper.

\bibliographystyle{abbrv}
\bibliography{allref_0330pm_5_3_2020}

\begin{thebibliography}{10}

\bibitem{BrSc:2008}
S.~C. Brenner and L.~R. Scott.
\newblock {\em The mathematical theory of finite element methods}, volume~15 of
  {\em Texts in Applied Mathematics}.
\newblock Springer, New York, third edition, 2008.

\bibitem{BrDoMa:1985}
F.~Brezzi, J.~Douglas, Jr., and L.~D. Marini.
\newblock Two families of mixed finite elements for second order elliptic
  problems.
\newblock {\em Numer. Math.}, 47(2):217--235, 1985.

\bibitem{ChCo:2012}
B.~Chabaud and B.~Cockburn.
\newblock Uniform-in-time superconvergence of {HDG} methods for the heat
  equation.
\newblock {\em Math. Comp.}, 81(277):107--129, 2012.

\bibitem{CoDiEr:2016}
B.~Cockburn, D.~A. Di~Pietro, and A.~Ern.
\newblock Bridging the hybrid high-order and hybridizable discontinuous
  {G}alerkin methods.
\newblock {\em ESAIM Math. Model. Numer. Anal.}, 50(3):635--650, 2016.

\bibitem{CoFuSa:2017}
B.~Cockburn, G.~Fu, and F.~J. Sayas.
\newblock Superconvergence by {$M$}-decompositions. {P}art {I}: {G}eneral
  theory for {HDG} methods for diffusion.
\newblock {\em Math. Comp.}, 86(306):1609--1641, 2017.

\bibitem{CoFuHuJiSaSa:2018}
B.~Cockburn, Z.~Fu, A.~Hungria, L.~Ji, M.~A. S\'{a}nchez, and F.-J. Sayas.
\newblock Stormer-{N}umerov {HDG} methods for acoustic waves.
\newblock {\em J. Sci. Comput.}, 75(2):597--624, 2018.

\bibitem{CoGoNgPeSa:2011}
B.~Cockburn, J.~Gopalakrishnan, N.~C. Nguyen, J.~Peraire, and F.-J. Sayas.
\newblock Analysis of {HDG} methods for {S}tokes flow.
\newblock {\em Math. Comp.}, 80(274):723--760, 2011.

\bibitem{CoGoSa:2010}
B.~Cockburn, J.~Gopalakrishnan, and F.-J. Sayas.
\newblock A projection-based error analysis of {HDG} methods.
\newblock {\em Math. Comp.}, 79(271):1351--1367, 2010.

\bibitem{CoMu:2015}
B.~Cockburn and K.~Mustapha.
\newblock A hybridizable discontinuous {G}alerkin method for fractional
  diffusion problems.
\newblock {\em Numer. Math.}, 130(2):293--314, 2015.

\bibitem{CoQu:2014}
B.~Cockburn and V.~Quenneville-B\'{e}lair.
\newblock Uniform-in-time superconvergence of the {HDG} methods for the
  acoustic wave equation.
\newblock {\em Math. Comp.}, 83(285):65--85, 2014.

\bibitem{DiDr:2017}
D.~A. Di~Pietro and J.~Droniou.
\newblock A hybrid high-order method for {L}eray-{L}ions elliptic equations on
  general meshes.
\newblock {\em Math. Comp.}, 86(307):2159--2191, 2017.

\bibitem{DiDr:2020}
D.~A. Di~Pietro and J.~Droniou.
\newblock {\em The Hybrid High-Order Method for Polytopal Meshes}.
\newblock Modeling, Simulation and Applications. Springer, Cham, 2020.
\newblock Design, Analysis, and Applications.

\bibitem{DuSa:2019}
S.~Du and F.-J. Sayas.
\newblock {\em An invitation to the theory of the hybridizable discontinuous
  {G}alerkin method}.
\newblock SpringerBriefs in Mathematics. Springer, Cham, 2019.
\newblock Projections, estimates, tools.

\bibitem{DuSa_elas:2020}
S.~Du and F.-J. Sayas.
\newblock New analytical tools for {HDG} in elasticity, with applications to
  elastodynamics.
\newblock {\em Math. Comp.}, 89(324):1745--1782, 2020.

\bibitem{ErGu:2017}
A.~Ern and J.-L. Guermond.
\newblock Finite element quasi-interpolation and best approximation.
\newblock {\em ESAIM Math. Model. Numer. Anal.}, 51(4):1367--1385, 2017.

\bibitem{GrMo:2011}
R.~Griesmaier and P.~Monk.
\newblock Error analysis for a hybridizable discontinuous {G}alerkin method for
  the {H}elmholtz equation.
\newblock {\em J. Sci. Comput.}, 49(3):291--310, 2011.

\bibitem{HuPrSa:2017}
A.~Hungria, D.~Prada, and F.-J. Sayas.
\newblock H{DG} methods for elastodynamics.
\newblock {\em Comput. Math. Appl.}, 74(11):2671--2690, 2017.

\bibitem{LeSc:2010}
C.~Lehrenfeld.
\newblock Hybrid discontinuous {G}alerkin methods for solving incompressible
  flow problems.
\newblock {\em Rheinisch-Westfalischen Technischen Hochschule Aachen}, 2010.

\bibitem{Ne:1980}
J.-C. N\'{e}d\'{e}lec.
\newblock Mixed finite elements in {${\bf R}^{3}$}.
\newblock {\em Numer. Math.}, 35(3):315--341, 1980.

\bibitem{Ne:1986}
J.-C. N\'{e}d\'{e}lec.
\newblock A new family of mixed finite elements in {${\bf R}^3$}.
\newblock {\em Numer. Math.}, 50(1):57--81, 1986.

\bibitem{Oi:2014}
I.~Oikawa.
\newblock Hybridized discontinuous {G}alerkin method for convection-diffusion
  problems.
\newblock {\em Jpn. J. Ind. Appl. Math.}, 31(2):335--354, 2014.

\bibitem{QiShSh:2018}
W.~Qiu, J.~Shen, and K.~Shi.
\newblock An {HDG} method for linear elasticity with strong symmetric stresses.
\newblock {\em Math. Comp.}, 87(309):69--93, 2018.

\bibitem{QiSh:2016}
W.~Qiu and K.~Shi.
\newblock An {HDG} method for convection diffusion equation.
\newblock {\em J. Sci. Comput.}, 66(1):346--357, 2016.

\bibitem{RaTh:1977}
P.-A. Raviart and J.~M. Thomas.
\newblock A mixed finite element method for 2nd order elliptic problems.
\newblock In {\em Mathematical aspects of finite element methods ({P}roc.
  {C}onf., {C}onsiglio {N}az. delle {R}icerche ({C}.{N}.{R}.), {R}ome, 1975)},
  pages 292--315. Lecture Notes in Math., Vol. 606, 1977.

\end{thebibliography}

\end{document}